\documentclass{amsart}

\usepackage{amssymb}
\usepackage{tikz}
\usetikzlibrary{calc}
\usetikzlibrary{decorations}
\usepackage{todonotes}
\usepackage{hyperref}
\hypersetup {
    colorlinks = true,
    linktoc = all,
    linkcolor = blue,
    citecolor = red
}

\newtheorem{theorem}{Theorem}[section]
\newtheorem{lemma}[theorem]{Lemma}
\newtheorem{corollary}[theorem]{Corollary}

\newtheorem{proposition}[theorem]{Proposition}
\newtheorem{conjecture}[theorem]{Conjecture}

\theoremstyle{definition}

\newtheorem{example}[theorem]{Example}

\theoremstyle{remark}
\newtheorem{remark}[theorem]{Remark}

\numberwithin{equation}{section}

\DeclareMathOperator{\Red}{Red}
\DeclareMathOperator{\uRed}{uRed}

\begin{document}

\title[Red sizes of quivers]{Red sizes of quivers}


\author{Eric Bucher}
\address{Department of Mathematics, Xavier University, Cincinnati, OH 45207 }
\email{buchere1@xavier.edu}

\author{John Machacek}
\address{Department of Mathematics, University of Oregon, Euegene, OR 97403}
\email{johnmach@uoregon.edu}
\thanks{J.M. with support from NSF Grant DMS-2039316.}




\begin{abstract}
In this article, we will expand on the notions of maximal green and reddening sequences for quivers associated to cluster algebras. The existence of these sequences has been studied for a variety of applications related to Fomin and Zelevinsky's cluster algebras. Ahmad and Li considered a numerical measure of how close a quiver is to admitting a maximal green sequence called a \emph{red number}.
In this paper we generalized this notion to what we call \emph{unrestricted red numbers} which are related to reddening sequences. In addition to establishing this more general framework we completely determine the red numbers and unrestricted red numbers for all finite mutation type quivers.
Furthermore, we give conjectures on the possible values of red numbers and unrestricted red numbers in general.
\end{abstract}

\maketitle

\section{Introduction}

Maximal green sequences were defined by Keller~\cite{Dilog} and play a role in the theory of Fomin and Zelevinsky’s cluster algebras~\cite{CA1}.
They also arise in representation theory, wall-crossing, and the computation of Donaldson-Thomas invariants. 
A survey of many results, properties, and the history of maximal green sequences can be found in~\cite{survey}. While the primary focus of the study of maximal green sequences has been to determine when such a sequence exists, a natural question arises; if a quiver does not admit a maximal green sequence then how close is it? 

Recently, Ahmad and Li have considered a numerical measure of how close a quiver is to possessing a maximal green sequence~\cite{genMGS}. This statistic is called the \emph{red number} of the quiver and measures how many of the vertices in the quiver can be turned red via a sequence of green mutations. We further study this measure, as well as put forth an analogous measure related to the more general notion of reddening sequences (which are also referred to as green-to-red sequences). This more general concept of \emph{unrestricted red numbers} measures how many of the vertices of a quiver can be turned red via any sequence of mutations. 

In this paper we will present the background definitions for quiver mutations in Section~\ref{sec:mu}. In Section~\ref{sec:redsize}, we will establish the key concepts of this article: red numbers and unrestricted red numbers of quivers. We will present some foundational results regarding these new concepts as well as present a conjecture about the nature of (unrestricted) red numbers for arbitrary quivers. In Section~\ref{sec:finite}, we will explicitly find the red numbers and unrestricted red numbers for quivers of finite mutation type, verifying the conjecture for finite mutation type quivers. These quivers play a vital role in the classification of cluster algebras, and the results will give definitive answers for this important family of quivers. We conclude in Section~\ref{sec:conclusion} with providing some analysis of red number on quivers with few vertices and discussion of possible approaches to our conjectures using scattering diagrams and universal quivers.

\section{Quiver mutation}\label{sec:mu}
A \emph{quiver} is a pair $Q = (V,E)$ where $V$ is a set of \emph{vertices} and $E$ is mulitset \emph{arrows} between two distinct vertices.
That is, if $i \to j$ for $i,j \in V$ is an arrow then we require that $i \neq j$ (i.e. no loops).
Furthermore, we do not allow both $i \to j \in E$ and $j \to i \in E$ (i.e. no $2$-cycles).
Let us point out that $E$ is a multiset meaning multiple arrows between to particular vertices is permitted.
Given a quiver $Q = (V,E)$ the corresponding \emph{framed} quiver is $\hat{Q} = (V \sqcup F, E \sqcup E')$ where $F = \{i' : i \in V\}$ and $E' = \{i \to i' : i \in V\}$.
In other words the framed quiver obtained to adding a new vertex $i'$ for each $i \in V$ and then adding an arrow $i \to i'$.
We retain the decomposition of the vertex set $V \sqcup F$.
All elements of $V \sqcup F$ are vertices, but vertices in $V$ are called \emph{mutable} while vertices in $F$ are call \emph{frozen}.
The notion of a framed quiver will be essential the to defining the red and green mutation which allows us to consider maximal green and reddening sequences.

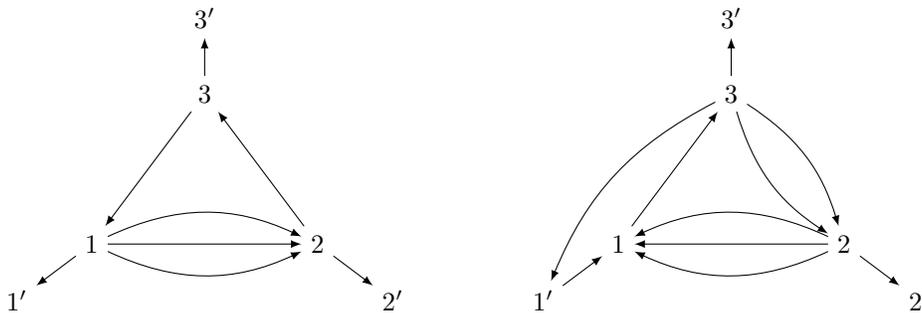
\begin{figure}
    \centering
\begin{tikzpicture}
    \node (1) at (-1.5,0) {$1$};
    \node (2) at (1.5,0) {$2$};
    \node (3) at (0,2) {$3$};
    \node (11) at (-2.5,-0.75) {$1'$};
    \node (22) at (2.5,-0.75) {$2'$};
    \node (33) at (0,3) {$3'$};
    \draw[-{latex}] (1) to (11); 
    \draw[-{latex}] (2) to (22); 
    \draw[-{latex}] (3) to (33);
    \draw[-{latex}, bend right=25] (1) to (2);
    \draw[-{latex}] (1) to (2);
    \draw[-{latex}, bend left=25] (1) to (2);
    \draw[-{latex}] (2) to (3);
    \draw[-{latex}] (3) to (1);
    
    \begin{scope}[shift={(7,0)}]
    \node (1) at (-1.5,0) {$1$};
    \node (2) at (1.5,0) {$2$};
    \node (3) at (0,2) {$3$};
    \node (11) at (-2.5,-0.75) {$1'$};
    \node (22) at (2.5,-0.75) {$2'$};
    \node (33) at (0,3) {$3'$};
    \draw[-{latex}, bend right=20] (3) to (11);
    \draw[-{latex}] (11) to (1); 
    \draw[-{latex}] (2) to (22); 
    \draw[-{latex}] (3) to (33);
    \draw[-{latex}, bend right=25] (2) to (1);
    \draw[-{latex}] (2) to (1);
    \draw[-{latex}, bend left=25] (2) to (1);
    \draw[-{latex}, bend right=20] (3) to (2);
    \draw[-{latex},  bend left=20] (3) to (2);

    \draw[-{latex}] (1) to (3);
    
    \end{scope}

\end{tikzpicture}
    \caption{A framed quiver $\hat{Q}$ on the left and $\mu_1(\hat{Q})$ the result of a mutation on the right.}
    \label{fig:mutation}
\end{figure}

We are now ready to define the process of quiver mutation which is a combinatorial ingredient in Fomin and Zelevinsky’s cluster algebras~\cite{CA1}. 
The \emph{mutation} of a quiver $Q$ (framed quiver $\hat{Q}$) at a mutable vertex $k$ is denoted $\mu_k(Q)$ ($\mu_k(\hat{Q})$) and is obtained by
\begin{enumerate}
    \item adding an arrow $i \to j$ for any $2$-path of vertices $i \to k \to j$,
    \item reverse all arrows incident to $k$,
    \item delete a maximal collection of $2$-cycles as well as any arrow between frozen vertices.
\end{enumerate}
An example of a quiver and mutation in shown in Figure~\ref{fig:mutation}.

Let $i$ be a mutable vertex in $\mu_{j_{\ell}} \circ \cdots \circ \mu_{j_2} \circ \mu_{j_1}(\hat{Q})$ for any sequence $j_1, j_2, \dots, j_{\ell}$ of mutable vertices.
The vertex $i$ is called \emph{green} if there exists a frozen vertex $j'$ and arrow $i \to j'$.
The vertex $i$ is called \emph{red} if there exists a frozen vertex $j'$ and arrow $j' \to i$.
By a property known as \emph{sign coherence} a vertex is always green or red and never both.
This is an important property in cluster algebra theory that was conjectured by Fomin and Zelevinsky~\cite{CA4}.
It was first proven to hold by Derksen, Weyman, and Zelevinsky~\cite{DWZ} (see~\cite{Nagao} for another proof and~\cite{GHKK} for a proof in a case of greater generality).

A sequence of vertices $(i_1, i_2, \dots, i_{\ell})$ of a quiver $Q$ is called a \emph{reddening sequence} if all vertices
of $\mu_{\ell} \circ \cdots \circ \mu_{i_2} \circ \mu_{i_1} ( \hat{Q})$ are red.
If in addition for all $1 \leq j \leq \ell$ the vertex $i_j$ is green in $\mu_{j-1} \circ \cdots \circ \mu_{i_2} \circ \mu_{i_1} ( \hat{Q})$, then the sequence $(i_1, i_2, \dots, i_{\ell})$ is called a \emph{maximal green sequence}.
The definition of a maximal green sequence is due to Keller~\cite{Dilog}.
These sequences of mutations are purely combinatorial with the definition above, but they have numerous applications in cluster algebra theory and representation theory~\cite{survey}.
A sequence of vertices $(i_1, i_2, \dots, i_{\ell})$ so that $i_j$ is green in $\mu_{j-1} \circ \cdots \circ \mu_{i_2} \circ \mu_{i_1} ( \hat{Q})$ for all $1 \leq j \leq \ell$ is called a \emph{green sequence}.
With a green sequence there is no condition on the final color of the vertices after apply the mutations.

A quiver $Q$ is said to be of \emph{finite mutation type} if the set of all quivers that can be obtain from $Q$ by iteratively applying mutation is finite.
A classification of quivers of finite mutation type was completed by Felikson, Shapiro, and Tumarkin~\cite{FSTu}.
The main source of quivers of finite mutation type is adjacency quivers of ideal triangulations of bordered marked surfaces~\cite{FSTh}.
A triangulation of the once punctured torus and its adjacency quiver that is known as the \emph{Markov quiver} is shown in Figure~\ref{fig:Markov}.
The remaining quivers of finite mutation type either have two vertices or come from 11 exceptional types.
The exceptional quiver of finite mutation type that will be of the most interest to us is $X_7$ which is shown in Figure~\ref{fig:X7} and originally found by Derksen and Owen~\cite{DO}.

\begin{figure}
    \centering
    \begin{tikzpicture}
    
    \draw[red, ultra thick] (-1.5,1.5) -- (1.5,1.5);
    \draw[red, ultra thick] (-1.5,-1.5) -- (1.5,-1.5);
    \draw[blue, ultra thick] (-1.5,1.5) -- (-1.5,-1.5);
    \draw[blue, ultra thick] (1.5,1.5) -- (1.5,-1.5);
    \node[draw,circle,fill=gray,scale=0.5] (a) at (0,0) {};
    \draw (a) -- (-1.5,0);
    \draw (1.5,0) -- (a);
    \draw[purple] (a) -- (1.5,1.5);
    \draw[purple] (-1.5,-1.5) -- (a);
    \draw[black!60!green] (a) -- (0,1.5);
    \draw[black!60!green] (0,-1.5) -- (a);
    \node at (0.75,-0.2) {$1$};
    \node at (-0.2,0.75) {$2$};
    \node at (-0.69, -0.89) {$3$};
    
    \node (1) at (5,1) {$1$};
    \node (2) at (3.5,-1) {$2$};
    \node (3) at (6.5,-1) {$3$};
    \draw[-{latex}, bend left=10] (1) to (2);
    \draw[-{latex}, bend right=10] (1) to (2);
    \draw[-{latex}, bend left=10] (2) to (3);
    \draw[-{latex}, bend right=10] (2) to (3);
    \draw[-{latex}, bend left=10] (3) to (1);
    \draw[-{latex}, bend right=10] (3) to (1);
 \end{tikzpicture}
    
    \caption{An ideal triangulation of the once punctured torus on the left and its adjacency quiver on the right.}
    \label{fig:Markov}
\end{figure}
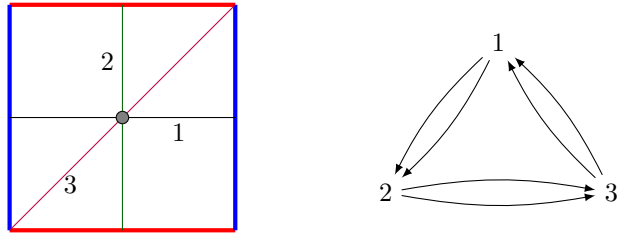

\section{Red sizes of quivers}\label{sec:redsize}

\subsection{Red sizes} We now give two definitions originally proposed by Ahmad and Li in~\cite{genMGS}.
A \emph{general maximal green sequence} for a quiver $Q$ is a green sequence $(i_1, i_2, \dots, i_{\ell})$ of vertices of $Q$ so that $\mu_{i_{\ell}} \cdots \mu_{i_2} \mu_{i_1}(\widehat{Q})$ has that maximal possible number of red vertices among any possible applying green sequence of mutations.
The \emph{red size} of a quiver $Q$ is the number of red vertices in $\mu_{i_{\ell}} \circ \cdots \circ \mu_{i_2} \circ \mu_{i_1}(\widehat{Q})$ for any general maximal green sequence $(i_1, i_2, \dots, i_{\ell})$.
We let $\Red(Q)$ denote the red size of $Q$.

Let us make the analogous definitions for reddening sequences.
A \emph{general reddening sequence} for a quiver $Q$ is a sequence $(i_1, i_2, \dots, i_{\ell})$ of vertices of $Q$ so that $\mu_{i_{\ell}} \cdots \mu_{i_2} \mu_{i_1}(\widehat{Q})$ has that maximal possible number of red vertices obtainable by applying mutations to $Q$.
The \emph{unrestricted red size} of a quiver $Q$ is the number of red vertices in $\mu_{i_{\ell}} \circ \cdots \circ \mu_{i_2} \circ \mu_{i_1}(\widehat{Q})$ where  $(i_1, i_2, \dots, i_{\ell})$ is a general reddening sequence.
We let $\uRed(Q)$ denote the unrestricted red size of $Q$.

We briefly discuss some basic facts which follow immediately from these definitions.
First, a quiver $Q$ admits a maximal green sequence if and only if $\Red(Q) = |V(Q)|$. 
Similarly, a quiver $Q$ admits a reddening sequence if and only if $\uRed(Q) = |V(Q)|$.  
Also, for any $Q$ we have that $\Red(Q) \leq \uRed(Q)$.
Equality $\Red(Q) = \uRed(Q)$ can hold (e.g. any quiver which admits a maximal green sequence, take for instance an acyclic quiver~\cite[Lemma 2.20]{BDP}).
Also, strict inequality is possible (e.g. any quiver which admits a reddening sequence but not a maximal green sequence, take for instance Muller's example~\cite[Figure 1]{MullerEJC}).

\subsection{Triangular extension} For any two quivers $Q_1$ and $Q_2$ a \emph{triangular extension} of $Q_1$ by $Q_2$ is any quiver $Q$ with
\[V(Q) = V(Q_1) \sqcup V(Q_2)\]
\[E(Q) = E(Q_1) \sqcup E(Q_2) \sqcup E\]
where $E$ is any set of arrows such that $i \to j \in E$ implies $i \in Q_1$ and $j \in Q_2$.
The triangular extension plays on important role in the existence and construction of both maximal green and reddening sequences~\cite{BanffP, CP, Uniform, typeA}.
These techniques can be adapted to construct general maximal green sequences or at least bound the red number.
However, the case of general maximal green sequences lacks a desirable property compared with the case where a maximal green sequence exists.

Let $Q$ be a triangular extension of $Q_1$ by $Q_2$.
It is known that in this case $\Red(Q) \geq \Red(Q_1) + \Red(Q_2)$~\cite[Theorem 4.8]{genMGS}.
Furthermore, it is true that $Q$ admits a maximal green sequence if and only if both $Q_1$ and $Q_2$ admit maximal green sequences~\cite[Theorem 4.10]{Uniform}.
It can be shown in a similar fashion that also $\uRed(Q) \geq \uRed(Q_1) + \uRed(Q_2)$ (cf. \cite[Remark 4.6]{Uniform}).
Ahmad and Li and have asked if $\Red(Q) = \Red(Q_1) + \Red(Q_2)$ when $Q$ is a triangular extension of $Q_1$ by $Q_2$~\cite[Question 4.9]{genMGS}.
We answer this question, as well as the analogous question for $\uRed$, in the negative.
\begin{figure}
    \centering
\begin{tikzpicture}[scale=0.5]
\node (1) at (-1,0) {$1$};
\node (2) at (-4,2) {$2$};
\node (3) at (-4,-2) {$3$};
\node (4) at (1,0) {$4$};
\node (5) at (4,2) {$5$};
\node (6) at (4,-2) {$6$};
\draw[-{latex}, bend left=10] (1) to (2);
\draw[-{latex}, bend right=10] (1) to (2);
\draw[-{latex}, bend left=10] (2) to (3);
\draw[-{latex}, bend right=10] (2) to (3);
\draw[-{latex}, bend left=10] (3) to (1);
\draw[-{latex}, bend right=10] (3) to (1);
\draw[-{latex}] (1) to (4);
\draw[-{latex}, bend left=10] (4) to (5);
\draw[-{latex}, bend right=10] (4) to (5);
\draw[-{latex}, bend left=10] (5) to (6);
\draw[-{latex}, bend right=10] (5) to (6);
\draw[-{latex}, bend left=10] (6) to (4);
\draw[-{latex}, bend right=10] (6) to (4);
\end{tikzpicture}
    \caption{A triangular extension.}
    \label{fig:triExt}
\end{figure}
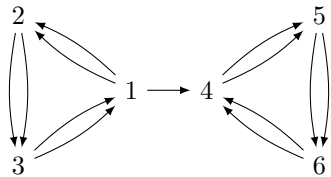

\begin{table}[]
    \centering
\begin{tabular}{|c|c|}\hline
General maximal green sequence & Remaining green vertex\\ \hline
$(3,4,1,2,4,6,5)$ & $1$ \\ \hline
$(1,4,2,3,5,4,6,5,6,2,5)$ & $2$ \\ \hline
$(3,2,1,4,5,3,6,5,6,3,5,3,5,6)$ & $3$ \\ \hline
$(3,1,4,2,6,5)$ & $4$ \\ \hline
$(3,1,4,2,6,5,4)$ & $5$ \\ \hline
$(6,5,4,2,3,1,3,4,1,6,3,4)$ & $6$ \\ \hline
\end{tabular}
    \caption{General maximal green sequences for the quiver in Figure~\ref{fig:triExt}.}
    \label{tab:triExt}
\end{table}

\begin{proposition}
There exists quivers $Q$, $Q_1$, and $Q_2$ such that $Q$ is a triangular extension of $Q_1$ by $Q_2$ and $\Red(Q) > \Red(Q_1) + \Red(Q_2)$.
Similarly, there exists quivers $Q$, $Q_1$, and $Q_2$ such that $Q$ is a triangular extension of $Q_1$ by $Q_2$ and $\uRed(Q) > \uRed(Q_1) + \uRed(Q_2)$.
\label{prop:triangular}
\end{proposition}
\begin{proof}
We will give one example which establishes the proposition for both $\Red$ and $\uRed$.
Let $Q$ be the quiver in Figure~\ref{fig:triExt}  which is a triangular extension of $Q_1 = Q|_A$ and $Q_2 = Q|_B$ for $A = \{1,2,3\}$ and $B = \{4,5,6\}$.
Here we have that
\[\Red(Q_1) = \uRed(Q_1) = \Red(Q_2) = \uRed(Q_2) = 2\]
since $Q|_A = Q|_B$ is the Markov quiver which is known not the admit a reddening sequence~\cite[Remark 3.4]{Lad}.
However, we find that $\Red(Q) = \uRed(Q) = 5$ since any one of the general maximal green sequences in Table~\ref{tab:triExt} produces $5$ red vertices.
\end{proof}

We have found in Proposition~\ref{prop:triangular} that red sizes do not respect triangular extensions in the same way that maximal green sequences do.
In the case of red sizes we can end up with more red vertices than one might expect.
This observation leads directly into the next subsection where we conjecture it is always possible to find a sequence of mutations resulting in only $1$ remaining green vertex for any connected quiver.

\subsection{Conjectures} 

Let us now make two conjectures.
The first conjecture is in the spirit of Muller's result~\cite[Corollary 3.2.2.]{MullerEJC} and deals with the potential mutation invariance of the unrestricted red number.

\begin{conjecture}
The unrestricted red size is mutation invariant. That is, $\uRed(Q) = \uRed(\mu_k(Q))$ for any $k$.
\label{conj:muinv}
\end{conjecture}

Since the existence of a maximal green sequence is not mutation invariant~\cite{MullerEJC}, the above conjecture is false if $\uRed(Q)$ is replaced by $\Red(Q)$.
We now make the a second conjecture (with multiple versions of varying strength) and establishing special cases of the this conjecture will be the focus of the rest of the paper.
Even the weakest version of the following conjecture implies Conjecture~\ref{conj:muinv}.

\begin{conjecture}
For any connected quiver $Q$ which does not admit a reddening sequence
\begin{itemize}
\item[(i)]we have $\Red(Q) = |V(Q)| - 1$.
\item[(ii)]we have $\Red(Q) = |V(Q)| - 1$ and every vertex of $Q$ is the last remaining green vertex after the application of some general maximal green sequence.
\item[(iii)]we have $\uRed(Q) = |V(Q)| - 1$.
\item[(iv)] we have $\uRed(Q) = |V(Q)| - 1$ and every vertex of $Q$ is the last remaining green vertex after the application of some general reddening sequence.
\end{itemize}
\label{conj:1green}
\end{conjecture}

If $Q$ is the disjoint union of two quivers which do no admit a reddening sequences then $\Red(Q) \leq \uRed(Q) \leq |V(Q)| - 2$.
So, the connectedness hypothesis is needed in Conjecture~\ref{conj:1green} otherwise one must treat connected components of the quiver individually.
Also, as previously stated it is the case that Conjecture~\ref{conj:1green} implies Conjecture~\ref{conj:muinv}.
To see this start by assuming part (iii) of Conjecture~\ref{conj:1green}, which is the weakest version of the conjecture, implying that $\uRed(Q) = |V(Q)|$ or $\uRed(Q) = |V(Q)|-1$ for any connected quiver $Q$.
Note for Conjecture~\ref{conj:muinv} we can restrict the connected quivers with out loss of generality.
So, the quiver $Q$ either admits a reddening sequence and has $\uRed(Q) = |V(Q)|$ or else does not admit a reddening sequence and has $\uRed(Q) = |V(Q)| - 1$.
Since the existence of a reddening sequence is mutation invariant this would imply that the unrestricted red size is also mutation invariant.

From Table~\ref{tab:triExt} we can verify that part (ii) of Conjecture~\ref{conj:1green}, which is the strongest version of the conjecture, holds for the quiver in Figure~\ref{fig:triExt}.
These general maximal green sequences were found with ad hoc methods, but we believe this example is instructive.
One thing to observe about the quiver $Q$ in Figure~\ref{fig:triExt} is that in $Q \setminus \{v\}$ contains the Markov quiver as on induced subquiver for any vertex $v$.
It then follows that $Q \setminus \{v\}$ does not have a maximal green sequence~\cite[Theorem 1.4.1]{MullerEJC}.
So, to find a general maximal green sequence leaving only the vertex $v$ green we must mutate at $v$ at least once.

\section{Finite mutation type and red size}\label{sec:finite}
It is known exactly which quivers of finite mutation type admit a maximal green sequence, which turns out to be equivalent to admitting a reddening sequence for this class of quivers, by the classification completed by Mills~\cite{Mills}.
A quiver $Q$ is of finite mutation type, then $\Red(Q) = |V(Q)|$ if and only if $Q$ is not mutation equivalent to $X_7$~\cite{X7} nor a quiver from a once punctured closed marked surface~\cite{Lad}.

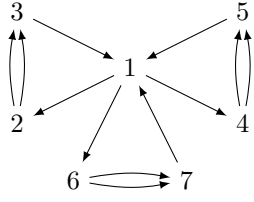
\begin{figure}
    \centering
\begin{tikzpicture}[scale=0.5]
\node (1) at (0,0) {$1$};
\node (2) at (-3,-1.5) {$2$};
\node (3) at (-3,1.5) {$3$};
\node (4) at (3,-1.5) {$4$};
\node (5) at (3,1.5) {$5$};
\node (6) at (-1.5,-3) {$6$};
\node (7) at (1.5,-3) {$7$};

\draw[-{latex}] (1) to (2);
\draw[-{latex}, bend left=10] (2) to (3);
\draw[-{latex}, bend right=10] (2) to (3);
\draw[-{latex}] (3) to (1);

\draw[-{latex}] (1) to (4);
\draw[-{latex}, bend left=10] (4) to (5);
\draw[-{latex}, bend right=10] (4) to (5);
\draw[-{latex}] (5) to (1);

\draw[-{latex}] (1) to (6);
\draw[-{latex}, bend left=10] (6) to (7);
\draw[-{latex}, bend right=10] (6) to (7);
\draw[-{latex}] (7) to (1);
\end{tikzpicture}
    \caption{The quiver $X7$}.
    \label{fig:X7}
\end{figure}

\begin{table}[]
    \centering
\begin{tabular}{|c|c|}\hline
General maximal green sequence & Remaining green vertex\\ \hline
$(2,4,6,3,5,7)$ & $1$ \\ \hline
$(2,4,5,3,7,1,3,5,2,1,4,3,5,1)$ & $6$ \\ \hline
$(1,3,2,3,5,2,3,6,4,2,6,4,3,1,2,4,5)$ & $7$\\ \hline
\end{tabular}
    \caption{General maximal green sequences for $X_7$.}
    \label{tab:X7}
\end{table}

\begin{table}[]
    \centering
\begin{tabular}{|c|c|}\hline
General maximal green sequence & Remaining green vertex\\ \hline
$(3,5,7,2,4,6,3,5,7)$ & $1$ \\ \hline
$(1,2,4,5,3,7,1,3,5,2,1,4,3,5)$ & $6$ \\ \hline
$(3,2,3,5,2,3,6,4,2,6,4,3,1,2,4,5,2)$ & $7$\\ \hline
\end{tabular}
    \caption{General maximal green sequences for $\mu_1(X_7)$.}
    \label{tab:muX7}
\end{table}

\begin{lemma}
If $Q$ is mutation equivalent to $X_7$, then
\[\Red(Q) = \uRed(Q) = 6 = |V(X_7)| - 1\]
and every vertex of $X_7$ is the last remaining green vertex after the application of some general maximal green sequence.
\label{lem:X7}
\end{lemma}
\begin{proof}
It is known that neither of the two quivers in the mutation class of $X_7$ admit a reddening sequence.
So, it suffices to provide a green sequence for each vertex $v$ of both quivers in the mutation class of $X_7$ that results in $6$ red vertices with $v$ being that last remaining red vertex.
Let us consider $X_7$ with vertices as labeled in Figure~\ref{fig:X7}.
General maximal green sequences each of which has only one remaining green vertex are given in Table~\ref{tab:X7}.
After considering the automorphism group of the quiver $X_7$ we can conclude the lemma holds for $X_7$.
Mutating the quiver $X_7$ at vertex $1$ gives the only other quiver in the mutation class of $X_7$.
For this other quiver general maximal green sequences are shown in Table~\ref{tab:muX7}
\end{proof}

\begin{theorem}
If $Q$ is a quiver of finite mutation type which does not admit a reddening sequence, then
\[\Red(Q) = \uRed(Q) = |V(Q)| - 1\]
and every vertex of $Q$ is the last remaining green vertex after the application of some general maximal green sequence.
\end{theorem}

\begin{proof}
The only quivers of finite mutation type which do not admit reddening sequences are $X_7$ and quivers which arise from ideal triangulation of once punctured closed surfaces.
The case of $X_7$ is handled by Lemma~\ref{lem:X7}.
So, it remains to take a quiver $Q$ such that $Q$ is a adjacency quiver of an ideal triangulation of once punctured closed surface.

Let $(S,M)$ be a marked surface and $T$ any triangulation of this surface. 
Now let $Q_T$ denote the quiver associated to this triangulation. Let $\alpha$ be an arc in $T$, with $v_{\alpha}$ the associated vertex in $Q_T$. Since the surface $S$ does not admit a reddening sequence we know that it must be a closed surface with exactly one marked point. Then consider the surface obtained by replacing $\alpha$ with a boundary component with one marked point. One of two things will occur, either this \emph{cut} along the arc $\alpha$ with split the surface into two surfaces each with exactly one boundary component and one marked point on that component; or it will produce a surface with exactly one boundary component and one marked point on the boundary. See the Figures~\ref{fig:otriang}~and~\ref{fig:newtriang} for an example of the former type of topological operation.

\begin{figure}
\begin{center}
\begin{tikzpicture}[line cap=round,line join=round,x=1.0cm,y=1.0cm]
\clip(-4.3,-0.96) rectangle (5.66,6.3);

\draw[dashed][color=blue] (-1,4).. controls (0,2) ..(-1,0);

\draw[dashed][color=blue] (4,4).. controls (3,2) ..(4,0);

\draw[color=black] (-.75,3.5).. controls (1.5,2.5) ..(3.75,3.5);

\draw[color=black] (-.75,.5).. controls (1.5,1.5) ..(3.75,.5);

\draw[color=red] (1.25,2.75).. controls (1,1.85) ..(1.25,1.25);

\draw[dashed][color=red] (1.25,2.75).. controls (1.5,1.85) ..(1.25,1.25);

\draw (1.05,1.85) -- (-.35,2.5);
\draw (1.05,1.85) -- (-.35,1.5);
\draw (1.05,1.85) -- (3.35,2.5);

\begin{scriptsize}
\fill  (1.05,1.85) node {\large{$\bullet$}}; 
\fill  (1.35,3.5) node {\large{$\alpha$}};
\end{scriptsize}
\end{tikzpicture}
\end{center}
\caption{Original Triangulation
\label{fig:otriang}}
\end{figure}
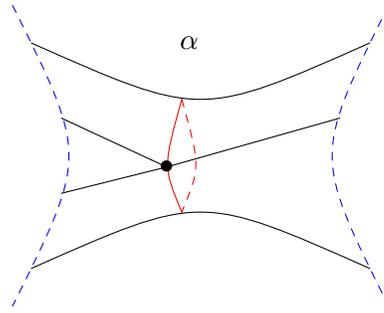

\begin{figure}
\begin{center}
\begin{tikzpicture}[line cap=round,line join=round,x=1.0cm,y=1.0cm]
\clip(-4.3,-0.96) rectangle (5.66,6.3);

\draw[dashed][color=blue] (-1,4).. controls (0,2) ..(-1,0);

\draw[dashed][color=blue] (4.5,4).. controls (3.75,2) ..(4.5,0);

\draw[color=black] (-.75,3.5)--(1.25,2.75);

\draw[color=black] (-.75,.5)--(1.25,1.25);

\draw[color=red] (1.25,2.75).. controls (1,1.85) ..(1.25,1.25);

\draw[color=red] (1.25,2.75).. controls (1.5,1.85) ..(1.25,1.25);

\draw[color=red] (2.25,2.75).. controls (2,1.85) ..(2.25,1.25);

\draw[color=red] (2.25,2.75).. controls (2.5,1.85) ..(2.25,1.25);

\draw (1.05,1.85) -- (-.35,2.5);
\draw (1.05,1.85) -- (-.35,1.5);

\draw (4.2,3.1) -- (2.25,2.75);
\draw (4.2,.75) -- (2.25,1.25);

\draw (2.45,1.85) -- (3.95,2.55);

\begin{scriptsize}
\fill  (1.05,1.85) node {\large{$\bullet$}}; 
\fill  (1.25,3.25) node {\large{$\alpha_1$}};
\fill  (2.45,1.85) node {\large{$\bullet$}}; 
\fill  (2.25,3.25) node {\large{$\alpha_2$}};
\end{scriptsize}
\end{tikzpicture}
\end{center}
\caption{Replacement Triangulation
\label{fig:newtriang}}
\end{figure}
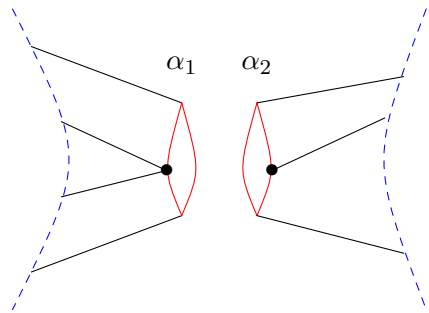

In either case, this produces a new triangulation of the surface $(S^{\alpha},M^{\alpha})$ which we will denote $T^{\alpha}$. Now we consider the associated quiver $Q_{T^{\alpha}}$. This is exactly the induced subquiver of $Q_T$ obtained by deleting the vertex $v_{\alpha}$.
Obtaining a quiver from a triangulation of different surface by deleting in this way is a known process, and can be done in more generality when considering surfaces that have boundary components and additional punctures (see e.g.~\cite[Lemma 9.10]{MullerAdv}). 

Since $Q_{T^{\alpha}}$ is a quiver associated to a triangulation of a surface with at least one boundary component (the one created in the above procedure) it admits a maximal green sequence by \cite{Mills}. We will denote the composition of mutations corresponding to this sequence bt $\sigma$.  

We will now apply $\sigma$ to $Q_T$.  This will be a sequence of green mutations, and the resulting quiver, $\sigma(Q_T)$, will have exactly one green vertex namely $v_{\alpha}$. Hence, $\uRed(Q_T) \geq \Red(Q_T) \geq n-1$ and we have $\uRed(Q_T) = \Red(Q_T) = n-1$.
Moreover, the vertex $v_\alpha$ that remains green can chosen arbitrarily and the theorem is proven.
\end{proof}

\begin{remark}
The general maximal green sequences constructed for $X_7$ and adjacency quivers of ideal triangulations of once punctured closed surfaces never mutate at the one vertex that remains green.
This means every single vertex deleted induced subquiver, and hence every proper induced subquiver, admits a maximal green sequence.
Therefore with respect to taking induced subquivers, finite mutation type quivers not admitting maximal green sequence are minimal examples of quivers not admitting maximal green sequences.
\end{remark}

We have the following immediate corollary since we have now verified Conjecture~\ref{conj:1green} for quivers of finite mutation type and have already discussed how the conjecture implies Conjecture~\ref{conj:muinv}.

\begin{corollary}
If $Q$ is a quiver of finite mutation type, then 
\[\Red(Q) = \uRed(Q) = \uRed(\mu_k(Q)) = \Red(\mu_k(Q))\]
for mutation in any direction $k$.
\end{corollary}

\section{Conclusion}\label{sec:conclusion}
We have proven Conjecture~\ref{conj:muinv} and the strongest version of Conjecture~\ref{conj:1green} for quivers of finite mutation type.
To conclude we discuss a few more cases our conjectures hold along with a few approaches to our conjectures in general.
Also, we indicate where some potential obstacles are encountered.

\subsection{Quivers with few vertices}
Now we verify that $\Red(Q) \geq |V(Q)| - 1$ for quivers with $1 \leq |V(Q)| \leq 4$.
In each case we are able to find a sufficiently large acyclic induced subquiver.

\begin{proposition}
If $|V(Q)| \leq 4$, then $\Red(Q) \geq |V(Q)|-1$.
\end{proposition}
\begin{proof}
If $|Q|=1$ or $|Q|=2$, then $Q$ is acyclic and has a maximal green sequence.
If $|Q|=3$, then choosing a maximal green sequence any induced subquiver on $2$ vertices will demonstrate that $\Red(Q) \geq 2$.
If $|Q| = 4$, then $Q$ must contain an induced acyclic subquiver on $3$ vertices.
A maximal green sequence for this induced acyclic subquiver on $3$ vertices will show $\Red(Q) \geq 3$.
\end{proof}

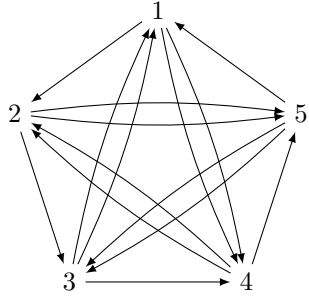
\begin{figure}
    \centering
    \begin{tikzpicture}
    \node (1) at ({2*sin(0*360/5)},{2*cos(0*360/5)}) {$1$};
    \node (2) at ({2*sin(-1*360/5)},{2*cos(-1*360/5)}) {$2$};
    \node (3) at ({2*sin(-2*360/5)},{2*cos(-2*360/5)}) {$3$};
    \node (4) at ({2*sin(-3*360/5)},{2*cos(-3*360/5)}) {$4$};
    \node (5) at ({2*sin(-4*360/5)},{2*cos(-4*360/5)}) {$5$};
    
    \draw[-{latex}] (1) to (2);
    \draw[-{latex}] (2) to (3);
    \draw[-{latex}] (3) to (4);
    \draw[-{latex}] (4) to (5);
    \draw[-{latex}] (5) to (1);
    \draw[-{latex},bend left=7] (1) to (4);
    \draw[-{latex},bend right=7] (1) to (4);
    \draw[-{latex},bend left=7] (4) to (2);
    \draw[-{latex},bend right=7] (4) to (2);
    \draw[-{latex},bend left=7] (3) to (1);
    \draw[-{latex},bend right=7] (3) to (1);
    \draw[-{latex},bend left=7] (2) to (5);
    \draw[-{latex},bend right=7] (2) to (5);
    \draw[-{latex},bend left=7] (5) to (3);
    \draw[-{latex},bend right=7] (5) to (3);

    \end{tikzpicture}

    \caption{The McKay quiver which has red size equal to 4.}
    \label{fig:McKay}
\end{figure}

Let us give one more small example involving a quiver which does not have a maximal green sequence, but for which we can verify we can mutate all but one vertex to red.

\begin{example}
In Figure~\ref{fig:McKay} we have the McKay quiver which is a quiver Br\"{u}stle,  Dupont, and P\'{e}rotin have shown does not admit a maximal green sequence~\cite[Example 8.2]{BDP}.
A general maximal green sequence for this quiver is
\[(1,4,2,5,3)\]
which results in only the vertex $2$ remaining green.
By the symmetry of the quiver we can conclude any vertex is the last remaining green vertex of some general maximal green sequence.
\end{example}

\subsection{Scattering diagrams and universal quivers}
\begin{figure}
    \centering
    \begin{tikzpicture}
    \node (u) at (-1,0) {$u$};
    \node (v) at (1,0) {$v$};
    \node (2) at (-1,-2) {$2$};
    \node (3) at (1,-2) {$3$};
    \node (1) at (-1,-4) {$1$};
    \node (4) at (1,-4) {$4$};
    \draw[-{latex}] (2) to (u);
    \draw[-{latex}] (3) to (v);
    \draw[-{latex}] (v) to (2);
    \draw[-{latex}] (u) to (3);
    
    \draw[-{latex}] (3) to (2);
    \draw[-{latex}, bend left = 17] (3) to (2);
    \draw[-{latex}, bend right = 17] (3) to (2);
    
    \draw[-{latex}, bend left = 7] (2) to (4);
    \draw[-{latex}, bend right = 7] (2) to (4);    
    \draw[-{latex}, bend left = 7] (1) to (3);
    \draw[-{latex}, bend right = 7] (1) to (3);
    
    \draw[-{latex}] (2) to (1);
    \draw[-{latex}] (4) to (1);
    \draw[-{latex}] (4) to (3);
    
    \begin{scope}[shift={(0.5,0.8660)},rotate=120]
        \node (u) at (-1,0) {};
    \node (v) at (1,0) {$w$};
    \node (2) at (-1,-2) {$6$};
    \node (3) at (1,-2) {$7$};
    \node (1) at (-1,-4) {$5$};
    \node (4) at (1,-4) {$8$};
    \draw[-{latex}] (2) to (u);
    \draw[-{latex}] (3) to (v);
    \draw[-{latex}] (v) to (2);
    \draw[-{latex}] (u) to (3);
    
    \draw[-{latex}] (3) to (2);
    \draw[-{latex}, bend left = 17] (3) to (2);
    \draw[-{latex}, bend right = 17] (3) to (2);
    
    \draw[-{latex}, bend left = 7] (2) to (4);
    \draw[-{latex}, bend right = 7] (2) to (4);    
    \draw[-{latex}, bend left = 7] (1) to (3);
    \draw[-{latex}, bend right = 7] (1) to (3);
    
    \draw[-{latex}] (2) to (1);
    \draw[-{latex}] (4) to (1);
    \draw[-{latex}] (4) to (3);
    \end{scope}
    
      \begin{scope}[shift={(-0.5,0.8660)},rotate=240]
        \node (u) at (-1,0) {};
    \node (v) at (1,0) {};
    \node (2) at (-1,-2) {$10$};
    \node (3) at (1,-2) {$11$};
    \node (1) at (-1,-4) {$9$};
    \node (4) at (1,-4) {$12$};
    \draw[-{latex}] (2) to (u);
    \draw[-{latex}] (3) to (v);
    \draw[-{latex}] (v) to (2);
    \draw[-{latex}] (u) to (3);
    
    \draw[-{latex}] (3) to (2);
    \draw[-{latex}, bend left = 17] (3) to (2);
    \draw[-{latex}, bend right = 17] (3) to (2);
    
    \draw[-{latex}, bend left = 7] (2) to (4);
    \draw[-{latex}, bend right = 7] (2) to (4);    
    \draw[-{latex}, bend left = 7] (1) to (3);
    \draw[-{latex}, bend right = 7] (1) to (3);
    
    \draw[-{latex}] (2) to (1);
    \draw[-{latex}] (4) to (1);
    \draw[-{latex}] (4) to (3);
    \end{scope}
    \end{tikzpicture}
    \caption{A $3$-universal quiver of Fomin, Igusa, and Lee~\cite{universal}.}
    \label{fig:universal}
\end{figure}
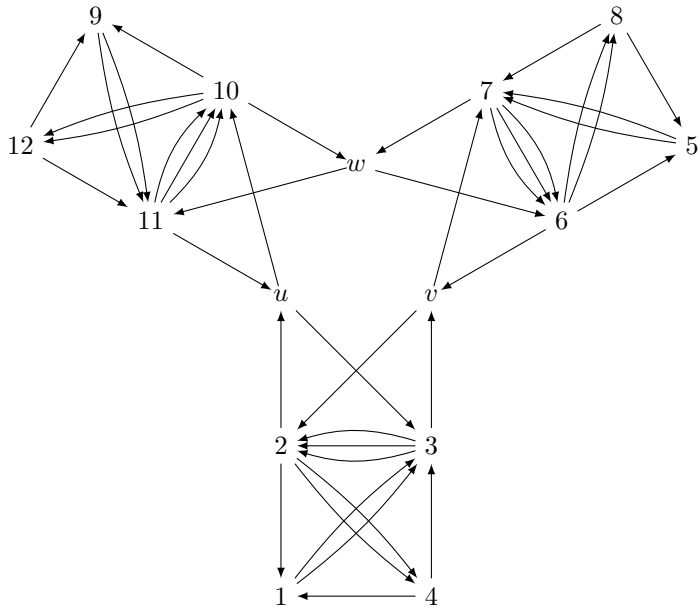

The (cluster) scattering diagrams of Gross, Hacking, Keel, and Kontsevich are a powerful tool that have been used to prove many important results in cluster algebra theory~\cite{GHKK}.
Using scattering diagrams Muller~\cite{MullerEJC} was able to show if a quiver admits a maximal green sequence (reddening sequence), then any induced subquiver also admits a maximal green (reddening sequence).
Additionally, it was established by Muller using scattering diagrams that the existence of a reddening sequences in a mutation invariant.
The fact that the chamber of the scattering diagram corresponding the when all vertices are red is distinguished as the all negative orthant was used in showing these results.

It would be desirable to show the mutation invariance in Conjecture~\ref{conj:muinv} as well as a statement of the form: If $\Red(Q) \geq |V(Q)| - \alpha$ for some constant $\alpha$, then $\Red(Q') \geq |V(Q')| - \alpha$ whenever $Q'$ is an induced subquiver of $Q$ (or the analogous statement with $\uRed$ in place of $\Red$).
We have not been able to obtain such a result due in part to the fact that without a reddening sequence we do not end the all negative orthant.
Rather the corresponding path ends in some chamber spanned a collection of $g$-vectors we have less control over.
Moreover, what prevents a maximal green or reddening sequence from existing is complicated regions of the scattering diagram which contain infinitely many walls.
Such regions are still not fully understood even for quivers with few vertices (see e.g.~\cite[Section 6.7]{Nak}).

Provided we had such a statement with $\alpha=1$, we could approach proving of our conjecture by attempting to verify it on a \emph{universal collection} of quivers as defined by Fomin, Igusa, and Lee~\cite{universal}. A universal collection of quivers has the property that every quiver is an induced subquiver of a quiver mutation equivalent to some quiver in the collection.
In any case universal collections give a source of quivers which necessarily do not admit reddening sequences.
An example of a universal quiver in given in Figure~\ref{fig:universal}.

We can produce general maximal green sequences for this quiver leaving only one vertex green, but we do not have a systematic approach.
The resulting quiver after a general maximal green sequence can be complicated unlike the case for a maximal green sequence which is related to the initial quiver in a simple manner~\cite[Proposition 2.10]{BDP}.
In fact it empirically appears that by iteratively performing mutation at random green vertices will result in a single remaining green vertex.

\bibliographystyle{alphaurl}
\bibliography{refs}

\end{document}